\documentclass[11pt]{amsart}

\usepackage{amsmath}
\usepackage{amssymb}
\usepackage{fancybox}
\usepackage{eucal}
\usepackage{amsthm}
\usepackage{amscd}
\usepackage{wasysym}
\usepackage{url}
\usepackage{MnSymbol} %smybols for infinitary formulas
\usepackage{natbib}
\usepackage{url}

\usepackage{amssymb,}
\usepackage[]{amsmath, amsthm, amsfonts,graphicx, amscd,}
\usepackage[all,cmtip]{xy}
\usepackage{setspace, xspace}

\usepackage[colorlinks=true,urlcolor=blue,linkcolor=blue]{hyperref}

\newcommand{\T}[0]{k}

\newcommand{\dcf}[0]{DCF_0}
\newcommand{\rcf}[0]{RCF}
\newcommand{\acfa}[0]{ACFA}

\DeclareMathSymbol{\mlq}{\mathord}{operators}{``}
\DeclareMathSymbol{\mrq}{\mathord}{operators}{`'}
\DeclareMathOperator{\dom}{dom}

\DeclareMathOperator{\cl}{cl}
\DeclareMathOperator{\tp}{tp}
\DeclareMathOperator{\acl}{acl}

\newtheorem {thm}{Theorem}[section]

\newtheorem{prop}[thm]{Proposition}

\newtheorem{claim}[thm]{Claim}
\theoremstyle{remark}
\newtheorem{rem}[thm]{Remark}
\newtheorem{np*}{Non-Proof}

\theoremstyle{definition}
\newtheorem{defn}[thm]{Definition}

\numberwithin{subcase}{case}

\newcommand{\mc}{\mathcal }

\newcommand{\nonitalics}{}

\begin{document}

\title{Independence in computable algebra}

\author[M. Harrison-Trainor]{Matthew Harrison-Trainor}
\address{Group in Logic and the Methodology of Science\\
University of California, Berkeley\\
 USA}
\email{matthew.h-t@berkeley.edu}
\urladdr{\href{http://www.math.berkeley.edu/~mattht/index.html}{www.math.berkeley.edu/$\sim$mattht}}

\author[A. Melnikov]{Alexander Melnikov}
\address{The Institute of Natural and Mathematical Sciences\\
Massey University\\
 New Zealand}
\email{alexander.g.melnikov@gmail.com}
\urladdr{\href{https://dl.dropboxusercontent.com/u/4752353/Homepage/index.html}{https://dl.dropboxusercontent.com/u/4752353/Homepage/index.html}}

\author[A. Montalb\'an]{Antonio Montalb\'an}
%\revauthor{Montalb\'an, Antonio}
\address{Department of Mathematics\\
University of California, Berkeley\\
	USA}
\email{antonio@math.berkeley.edu}
\urladdr{\href{http://www.math.berkeley.edu/~antonio/index.html}{www.math.berkeley.edu/$^\sim$antonio}}

\thanks{The first author was partially supported by the Berkeley Fellowship and NSERC grant PGSD3-454386-2014.
The second author was supported by the Packard Foundation.
The third author was partially supported by the Packard Fellowship and NSF grant \# DMS-1363310.}

\begin{abstract}
 We give a sufficient condition for an algebraic structure to have a computable presentation with a computable basis and a computable presentation with no computable basis. 
We apply the condition to  differentially closed, real closed, and difference closed fields with the relevant notions of independence. To cover these classes of structures we introduce a new technique of \emph{safe extensions} that was not necessary for the previously known results of this kind.  We will then apply our techniques to derive new corollaries on the number of computable presentations of these structures.
 The condition also implies classical and new results on vector spaces, algebraically closed fields, torsion-free abelian groups and Archimedean ordered abelian groups.  
\end{abstract}


\maketitle

%\tableofcontents

\section{Introduction}
The main objects of this paper are computable algebraic structures. A countably infinite algebraic structure $\mathcal{A}$ is \emph{computable} (Mal{\textquotesingle}cev~\cite{Malcev61} and Rabin~\cite{Rabin60}) if it admits a labeling of its domain by natural numbers so that the operations on $\mathcal{A}$ become Turing computable upon the respective labels. Such a numbering is called a \emph{computable presentation}, a \emph{computable copy}, or a \emph{constructivization} of $\mc{A}$. Without loss of generality, we restrict ourselves to countable structures with domain $\omega$ (the natural numbers).
Examples of computably presented structures include recursively presented groups with decidable word problem (Higman~\cite{Higman61}) and  explicitly presented fields  (van der Waerden~\cite{vanderWaerden30} and Fr\"ohlich and Shepherdson \cite{FrohlichShepherdson56}).

\subsection{Independence with applications}
 Mal{\textquotesingle}cev and his mathematical descendants  were perhaps the first to realize the fundamental role  of various notions of independence in effective algebra, especially in the study of the number of computable presentations of structures. In his pioneering paper~\cite{Malcev62},  Mal{\textquotesingle}cev made an important observation:
\begin{center}
The additive group $\mathbb{V}^\infty \cong \bigoplus_{i \in \mathbb{N}} \mathbb{Q}$ has two computable presentations that are not computably isomorphic.
\end{center}
This effect had never been seen before, since algorithms had mostly been applied to finitely generated structures whose presentations are effectively unique.  Mal{\textquotesingle}cev noted that $\mathbb{V}^\infty$ clearly has  a  ``good'' computable presentation $\mc{G}$ that is  built upon a computable basis, and he constructed a  ``bad'' computable presentation $\mc{B}$ of $\mathbb{V}^\infty$ that has no computable basis. Clearly,  $\mc{G}$ is not computably isomorphic to $\mc{B}$ (written $\mc{G} \not \cong_{comp} \mc{B}$). 
 Essentially the same argument applies to the algebraically closed field $\mathbb{U}$ of infinite transcendence degree~\cite{MetakidesNerode79}.
Similarly,  manipulations with bases were used in the study of the number of computable copies in the contexts of torsion-free abelian groups~\cite{Dobritsa83, Nurtazin74b, Goncharov82} and ordered abelian groups~\cite{GoncharovLemppSolomon03} of infinite rank (though for ordered abelian groups, the existence of a ``bad'' copy is a new result appearing in this paper). The latter two examples are nontrivial, since the existence of a ``good'' copy is \emph{not} evident.
Nonetheless, in all these examples the ``good'' copy $\mc{G}$ and the ``bad'' copy $\mc{B}$ are isomorphic \emph{relative to} the halting problem, or $\Delta^0_2$-isomorphic. Goncharov~\cite{Goncharov82} showed that $\mc{G} \not\cong_{comp} \mc{B}$ and $\mc{G} \cong_{\Delta^0_2} \mc{B}$ imply there exist \emph{infinitely many} computable presentations of the structure up to computable isomorphism. Thus, in each case discussed above we get infinitely many effectively different presentations.

Notions of independence play a central role in the study of the combinatorial properties of effectively presented vector spaces and for other structures with an appropriate notion of independence. Such studies were quite popular in the 70's and 80's; the standard reference is the fundamental paper of Metakides and Nerode~\cite{MetakidesNerode77}, see also \cite{Dekker69, Dekker71a, Dekker71b,Shore78,Downey84} and, for applications to reverse mathematics, \cite{Simpson99}. Many results on subspaces of effectively presented vector spaces go thorough in the abstract setting of computable pregeometries (to be defined) -- see the survey \cite{DowneyRemmel98} of Downey and Remmel. A number of results true of vector spaces go through for an arbitrary pregeometry if we have access to a ``good'' presentation $\mc{G}$ with a computable basis. See, e.g., a recent paper of Conidis and Shore~\cite{ConidisShore} for a non-elementary illustration of this phenomenon.

\subsection{\texorpdfstring{The Mal{\textquotesingle}cev property}{The Malcev property}}  We would like to unify the known results and  extend them  to other classes of structures.
 The definition below is central to the paper. Throughout the paper,   $\mc{K}$ stands for  a class of computable structures that admits a notion of independence (we will shortly clarify the notions of ``independence'' and ``basis'').

\begin{defn}
A class $\mathcal{K}$ has the
\emph{ Mal{\textquotesingle}cev property}
if each member $\mathcal{M}$ of $\mathcal{K}$ of infinite dimension  has 
a computable presentation $\mc{G}$ with a computable basis and a computable presentation $\mc{B}$ with no computable basis  such that $\mc{B} \cong_{\Delta^0_2} \mc{G}$.
\end{defn}

Of course, we can also talk about a single structure having the  Mal$'$cev property, but it will be more natural to consider only classes since all of our applications will be to an entire class of structures.

In this paper we address:

\begin{center}
\underline{Question}: Which common algebraic classes have the  Mal{\textquotesingle}cev property?
\end{center}
%Whenever a class of commutative algebraic structures is considered, it is usually associated with the corresponding natural notion of independence.
Note that, if  $\mathcal{K}$  has the  Mal{\textquotesingle}cev property, then every $\mathcal{M} \in \mathcal{K}$ of infinite dimension has infinitely many computable copies up to computable isomorphism~\cite{Goncharov82}. 
The standard abstraction for independence is the notion of a \emph{pregeometry} that we briefly discuss in the next subsection.

\subsection{Effective pregeometries} A pregeometry on $\mathcal{M}$ is a  function $\cl: \mathcal{P}(M) \rightarrow \mathcal{P}(M)$,
where $\cl(A)$ should be thought as the ``span'' of $A$~\cite{Whitney35}. A pregeometry must satisfy several natural properties (e.g., $\cl (\cl(A)) = \cl(A)$) that will be stated in the preliminaries (see \S \ref{sec:pre}).  Until then, the reader may safely rely on her/his intuition. For instance, we can define the notions of dimension, rank and  basis  in terms of $\cl$.
Computable pregeometries have been intensively studied, see survey~\cite{DowneyRemmel98}. We, however, will be working in the weaker context of pregeometries that can merely be \emph{effectively enumerated}. 
This phenomenon is typically captured in terms of relatively intrinsically computably enumerable (r.i.c.e.)~families of relations (to be defined--see Definition \ref{def:rice-pre}). 
It is crucial for us that being r.i.c.e.~is equivalent to being definable in $\mathcal{L}_{\omega_1 \omega}^c$ logic in the language of the structure by an infinitary computable $\Sigma_1$-formula~\cite{AshKnightManasseSlaman89,Chisholm90}. Thus, we  will have a syntactical definition of $\cl$,  and in fact this definition will induce a pregeometry on every member of a class $\mathcal{K}$ of computable algebras. For convenience, we say that $\cl$ is a $r.i.c.e.$ \emph{pregeometry on $\mathcal{K}$}.

\subsection{The meta-theorem} Let $\mc{M}$ be a computable structure, $(\mc{M},\cl)$ a r.i.c.e.~pregeometry. 
 The \emph{independence diagram} $\mc{I}_{\mc{M}}({\bar{c}})$ of $\bar{c}$ in $\mc{M}$ is the collection of all existential formulas true of tuples independent over $\bar{c}$. 
We also say that independent tuples are  \emph{locally indistinguishable} if for every  $\exists$-formula true of one independent  tuple  we can find independent witnesses for this formula within the $\cl$-span of \emph{any} independent tuple (we will give the definition later--see Definition \ref{defn:loc}--but intuitively, it means that we would not be able to ``code''  a computably enumerable (c.e.) set using existential formulas into the $\cl$-span of independent tuples, since the $\cl$-span of any two independent tuples is in some sense the same).
 We will prove that the condition below is sufficient for producing a ``good'' computable presentation of $\mathcal{M}$ with a computable basis:

\medskip

\noindent  \textbf{Condition G:} Independent tuples are locally indistinguishable in $\mc{M}$ and for each $\mc{M}$-tuple $\bar{c}$, $\mc{I}_{\mc{M}}({\bar{c}})$ is computably enumerable uniformly in~$\bar{c}$.

\medskip

 For convenience, we will be using one more term which should be thought of as saying that independent types are non-principal. We say that \emph{dependent elements are dense in $\mc{M}$} if, whenever $\mc{M} \models \exists \bar{y} \psi(\bar{c}, \bar{y}, a)$ for a quantifier-free formula $\psi$, non-empty tuple $\bar{c}$, and $a \in \mc{M}$, there is a $b\in \cl(\bar{c})$ such that $\mc{M} \models \exists \bar{y} \psi(\bar{c}, \bar{y}, b)$. We may also assume that $\bar{c}$ contains at least $m$ independent elements, for some fixed $m$. This corresponds to localizing the pregeometry at a finite set.
 The next property is sufficient for having a ``bad'' computable presentation of $\mathcal{M}$ with no computable basis:

\medskip

\noindent \textbf{Condition B:} Dependent elements   are dense in $\mc{M}$.

\medskip

\noindent  Furthermore, our methods allow us to keep the isomorphisms $\Delta^0_2$. We summarize the above results in a theorem:

\begin{thm}\label{main} \nonitalics
Let $\mc{K}$ be a class of computable structures that admits a r.i.c.e.~pregeometry $\cl$.
If  each $\mc{M}$ in $\mc{K}$ of infinite dimension satisfies Conditions $G$ and $B$, then $\mc{K}$ has the  Mal{\textquotesingle}cev property.
\end{thm}

\noindent The (meta-)theorem above is the central technical tool of the paper that allows us to separate our proofs into an algebraic part and a part consisting of the effective combinatorics of r.i.c.e.~pregeometries. We prove the metatheorem in \S \ref{sec:main-thm}.

\subsection{Applications} %Most of our new applications cover existentially closed field-like structures that generalize the notion of an algebraically closed field to fields with extra operations (e.g., to ordered fields).
%Our metatheorem also covers most, if not all, known results of this nature that deal with r.i.c.e.~pregeometries. It also simplifies and clarifies the known proofs. These simplifications allow us to derive a new fact on computable ordered abelian groups that was not known before.   

That vector spaces and algebraically closed fields have the Mal{\textquotesingle}cev property is obvious. We discuss several non-trivial applications below. As we will note in the conclusion, there are indeed many more potential applications of our metatheorem. To keep the paper short, we give only five applications that (we think) are the most important ones.

 \subsubsection{Differentially closed fields} A \textit{differential field} \cite{Ritt50} is a field $K$ together with a derivation operator $\delta: K \to K$. A \emph{differentially closed field} is an existentially closed differential field. Robinson and Blum~\cite{Robinson59,Blum68} came up with an elegant first-order axiomatization for differentially closed fields of characteristic zero ($\dcf$). There has been a lot of work on model theory of differentially closed fields as described in~\cite{Marker06}. It is known that  $\dcf$ is complete, decidable, and has quantifier elimination.

In contrast to algebraically closed fields, we don't know any ``natural'' example of a non-trivial differentially closed field. However Harrington~\cite{Harrington74} showed that every computable differential field can be computably embedded into a \emph{computable presentation} of its differential closure. Thus, if we start with some differential field that we fully understand, we at least can \emph{effectively construct} its differential closure.  Differentially closed fields have some further nice computability-theoretic properties including the low property (see, e.g., \cite{MarkerMiller}).
%For instance,  Marker and Miller~\cite{} have recently announced that $DCF_0$ have the low property.

Differential fields admit a natural notion of independence called $\delta$-independence (to be defined).
The first new application of Theorem~\ref{main}, which we prove in \S \ref{DCF}, is:

\begin{thm}\label{th:diff} \nonitalics The class of computable differentially closed fields  of characteristic zero has the  Mal{\textquotesingle}cev property with respect to $\delta$-independence.
\end{thm}
\medskip

Our result, in a way, improves the result of Harrington since every computable differential closed field has an even ``nicer'' presentation in which $\delta$-dependence is decidable (however, we may lose the computable embedding of the original differential field).  Furthermore, if the $\delta$-rank of $\mc{M}$ is infinite, then  we can produce infinitely many effectively non-isomorphic computable presentations of $\mc{M}$.

\subsubsection{Difference closed fields}
 A \textit{difference field}~\cite{Cohn65} is a field together with a distinguished automorphism $\sigma$. 
A \emph{difference closed field} is an existentially closed difference field. The theory $\acfa$ of \textit{difference closed fields} is first-order axiomatizable, and
 the theories $\acfa$, $\acfa_0$, and $\acfa_p$  are decidable, see \cite[(1.4) of][]{ChatzidakisHrushovski99}.  For a detailed exposition of the model theory of difference fields, see \cite{ChatzidakisHrushovski99}.
The authors are unaware of any ``naturally defined'' algebraic example of a difference closed field. Nonetheless, it is not hard to \emph{effectively construct} an example of a difference closure of a given algebraically closed field with an automorphism. It can be done using, say, an effective variation of the Henkin construction~\cite{AshKnight00} or Ershov's Kernel Theorem~\cite{ErshovGoncharov00}. Difference fields admit a natural notion of \emph{transformal independence} (to be defined).
In \S \ref{sec:ACFA} we prove:

\begin{thm}\label{th:acfa}\nonitalics The class of computable difference closed fields has the  Mal{\textquotesingle}cev property with respect to transformal independence.
\end{thm}

\noindent We emphasize that we also get a new corollary on the number of computable copies of  difference closed fields of infinite transformal rank.
\medskip

\subsubsection{Real closed fields}
We assume that the reader is familiar with the definition of an ordered field. \emph{Real closed fields} are existentially closed ordered fields.  Tarski~\cite{Tarski48} showed that the theory $\rcf$ is complete, decidable, and has quantifier elimination.
Model-theoretic features of real closed fields admit a generalization called \emph{o-minimality}, see e.g.~\cite{VanDenDries98}. We note that o-minimality has recently been applied to solve an open problem in pure number theory~\cite{PilaWilkie06}. 

Computability-theoretic properties of ordered and real closed fields have been investigated in~\cite{ErshovGoncharov00,KnightLange13,Levin}. 
In \S \ref{RCF} we prove:

\begin{thm}\label{th:rcf}\nonitalics The class of computable real closed fields  has the  Mal{\textquotesingle}cev property with respect to the standard field-theoretic (or, equivalently, model-theoretic) algebraic independence.
\end{thm}

%(To Mattew: Did they know anything about computable categoricity? Perhaps they did. Who's ``they?'')

%??? Oscar Levin says that it was known that infinite transcendence degree means not computably categorical. Lemma 3.1 of his paper says that any finite transcendence degree ordered field is computably categorical ???

\subsubsection{Torsion-free abelian groups}  The results of Nurtazin~\cite{Nurtazin74b} Dobrica~\cite{Dobritsa83} and Goncharov~\cite{Goncharov82} mentioned above can be summarized in one theorem:

\begin{thm}\label{th:tfag}\nonitalics The class of computable torsion-free abelian groups  has the  Mal{\textquotesingle}cev property with respect to linear independence over~$\mathbb{Z}$.
\end{thm}

\noindent In contrast to the number of computable copies and existence of a ``good'' copy,  existence of a ``bad'' copy is a very recent fact that can be found in \cite{Melnikov}. We give a proof, using our metatheorem, in \S \ref{sec:tfag}.

\subsubsection{Archimedean ordered abelian groups} Using a relatively involved combinatorial argument, Goncharov, Lempp and Solomon~\cite{GoncharovLemppSolomon03} showed that every computable Archimedean ordered abelian group has a computable copy with a computable base and that in the case of infinite rank it has infinitely many effectively distinct computable presentations.
Using model-theoretic properties of $(\mathbb{R}, +, \leq)$, we extend their results in \S \ref{sec:oag}:

\begin{thm}\label{th:aoag}\nonitalics The class of computable Archimedean ordered abelian groups has the  Mal{\textquotesingle}cev property with respect to linear independence over~$\mathbb{Z}$.
\end{thm}

\noindent We note that the existence of a computable presentation with no computable basis is a new result.

\section{Preliminaries}

\subsection{Pregeometries}\label{sec:pre} A dependence relation often induces  a \emph{pregeometry}.
Let $X$ be a set and $\cl: \mc{P}(X) \to \mc{P}(X)$ a function on $\mc{P}(X)$. We say that $\cl$ is a \textit{pregeometry} if:
\begin{enumerate}%[label=(\arabic{*})]
	\item $A \subseteq \cl(A)$ and $\cl(\cl(A)) = \cl(A)$,
	\item $A \subseteq B \Rightarrow \cl(A) \subseteq \cl(B)$,
	\item (finite character) $\cl(A)$ is the union of the sets $\cl(B)$ where $B$ ranges over finite subsets of $A$, and
	\item (exchange principle) if $a \in \cl(A \cup \{b\})$ and $a \notin \cl(A)$, then $b \in \cl(A \cup \{a\})$.
\end{enumerate}
An operation which satisfies the first two properties is called a \textit{closure operator}.  Let $(X,\cl)$ be a pregeometry, and $A \subseteq X$. We say that:
\begin{itemize}
\item[i.] $A$ is \textit{closed} if $A = \cl(A)$;
\item[ii.]  $A$ \textit{spans} a set $B \subseteq A$ if $B \subseteq \cl(A)$;
\item[iii.]  $A \subseteq X$ is \textit{independent} if for all $a \in A$, $a \notin \cl(A \backslash \{a\})$, and  $A$ is \textit{dependent} otherwise;
\item[iv.]  $B$ is a \textit{basis} for $Y \subseteq X$ if $B$ spans $Y$ and is independent. 

\end{itemize}

 One can show that $B$ is a basis for $Y$ if, and only if, $B$ is a maximal independent set contained in $Y$. A standard argument shows that every set has a basis, and that every basis for $Y$ has the same cardinality. This cardinality is the \textit{dimension} of $Y$, written  $\dim(Y)$.  We can also generalize iii~and define the associated notion of independence over a subset $C$ of  $X$.  We cite \cite{Pillay96} for more information about pregeometries.

Because a pregeometry has finite character, all of the information of the pregeometry is captured in the relations (for each $n$) $x \in \cl(y_1,\ldots,y_n)$ on finite tuples.   We will say that $\cl$ is \textit{c.e.}~(\textit{computable}) if each of the relations $x \in \cl(\{y_1,\ldots,y_n\})$ are computably enumerable (computable, respectively) uniformly in $n$.

\begin{defn}\label{def:rice-pre} A pregeometry $\cl$ on a structure $\mc{M}$ is relatively intrinsically computably enumerable~(r.i.c.e.) if the relations $x \in \cl(\{y_1,\ldots,y_n\})$ are  c.e.\ uniformly in $n$ within, and relative to,  any presentation of $\mathcal{M}$.
\end{defn}
\noindent  All standard examples of pregeometries (vector spaces, fields, etc.)~are r.i.c.e.\ pregeometries.
As we have already mentioned in the introduction, 
it follows from~\cite{AshKnightManasseSlaman89} and~\cite{Chisholm90} that there is a tuple $\bar{d}$ in $\mc{M}$ such that these relations are uniformly defined by  computable infinitary $\Sigma^c_1$ formulas with parameters $\bar{d}$ (the uniformity comes from a small modification to the same proof -- see, for example, \cite{Montalban12}). See \cite{AshKnight00} for a background on computable infinitary logic $\mathcal{L}^c_{\omega_1 \omega}$.
The easy proposition below may help the reader to develop some intuition.

\begin{prop}\label{comp-pregeometry}\nonitalics
Let $(X,\cl)$ be a pregeometry, with $X \subseteq \omega$ a computable set and $\cl$ c.e. Then $(X,\cl)$ has a computable basis if and only if $\cl$ is computable.
\end{prop}
\begin{proof}
A standard argument shows that a computable pregeometry has a computable basis. Now suppose $B$ is a computable basis for $(X,\cl)$.
To decide whether $x_1, \ldots, x_n$ are dependent or not, find a finite set $A \subseteq B$ such that $\{x_1,\ldots,x_n\} \subseteq \cl(A)$. Then $x_1, \ldots, x_n$ are independent if, and only if, there exist  $a_1, \ldots, a_n \in A$ such that $$a_1, \ldots, a_n \in \cl (\{x_1, \ldots, x_n\} \cup (A \setminus \{a_1, \ldots, a_n\} )).$$
The latter is a c.e.\ property. Since for $\{x_1,\ldots,x_n\}$ being dependent is c.e.\ as $\cl$ is c.e., we conclude that determining whether or not a tuple is dependent is computable. From this, we can compute the closure relation.
 \end{proof}

%\begin{prop}\label{comp-pregeometry}\nonitalics
%Let $(\mc{M},\cl)$ be a c.e.\ pregeometry. Then $(\mc{M},\cl)$ has a computable basis if and only if $\cl$ is computable.
%\end{prop}
%
%
%\begin{proof}
%A standard argument shows that a computable pregeometry has a computable basis. Now suppose $B$ is a computable basis for $(\mc{M},\cl)$.
%To decide whether $x_1, \ldots, x_n$ are dependent or not, find a finite set $A \subseteq B$ such that $\{x_1,\ldots,x_n\} \subseteq \cl(A)$. Then $x_1, \ldots, x_n$ are independent if, and only if, there exist  $a_1, \ldots, a_n \in A$ such that $$a_1, \ldots, a_n \in \cl (\{x_1, \ldots, x_n\} \cup (A \setminus \{a_1, \ldots, a_n\} )).$$
%The latter is a $\Sigma^0_1$-property. Since for $\{x_1,\ldots,x_n\}$ being dependent is $\Sigma^0_1$ as $(\mc{M},\cl)$ is a c.e.\ pregeometry, we conclude that the property is $\Delta^0_1$, as desired.
 %\end{proof}

 \subsection{\texorpdfstring{The notion of $k$-dependence}{The notion of k-dependence}} %In this subsection we introduce a restricted notion of partial independence that unifies the previously known approaches \cite{,}. 
 We will be using a computable way of approximating a r.i.c.e.~pregeometry.
 We fix a r.i.c.e.\ pregeometry $(\mc{M},\cl)$ upon a computable infinite structure $\mc{M}$.
 Since the relations $x \in \cl(\{y_1,\ldots,y_n\})$ are r.i.c.e., we have % can write them as disjunctions of existential formulas over fixed parameters $\bar{d}$, say
\[ x \in \cl(\{y_1,\ldots,y_n\}) \Leftrightarrow \mc{M} \models \bigdoublevee_{k \in S_n} (\exists \bar{z}) \phi_k(\bar{d},x,y_1,\ldots,y_n,\bar{z}) \]
where $S_n$ are c.e.\ sets of indices for open formulas with parameters $\bar{d}$, given uniformly (using a standard forcing argument one can show that all the formulas $\phi_k$ indeed share the same tuple of parameters $\bar{d}$). To simplify our notation, we suppress $\bar{d}$ in $\phi_k$. We approximate the relations $x \in \cl(\{y_1,\ldots,y_n\})$  as follows.

\begin{defn}\label{def:clk}
We  say that \textit{$x$ is $\T$-dependent on $\{y_1,\ldots,y_n\}$} and write $x \in \cl_\T(\{y_1,\ldots,y_n\})$ if $x$ comes from among the first $\T$-many elements of $\mc{M}$ and
$$\mc{M} \models \bigdoublevee_{i \leq k, i \in S_n}(\exists \bar{z} \in \mc{M}\upharpoonright k) \phi_{i}(\bar{d},x,y_1,\ldots,y_n,\bar{z}) $$ 
(i.e., the witnesses $\bar{z}$ come from among the first $\T$ elements of $\mc{M}$).% If we do not require the parameters to come from the first $\T$ elements of $\mc{M}$, we write $x \in \overline{\cl}_\T(\{y_1,\ldots,y_n\}).$ 

\end{defn}

Similarly, we say that a set $X$ is \textit{$\T$-dependent} if for some $x \in X$, $x \in \cl_\T(X \backslash \{x\})$, and otherwise we say that $X$ is \textit{$\T$-independent}. It is clear that $\cl_k$
is a computable operator which, however,  typically is not a pregeometry. Moreover, $\cl_k(X)$ is finite for every set $X$.

\subsection{Locally indistinguishable tuples} %The \emph{existential type} of a tuple $\bar{a}$ in $\mc{M}$, perhaps with parameters, is the collection of all existential first-order formulas true of $\bar{a}$ in $\mc{M}$. We write $\exists$-$tp^{\bar{a}/\bar{c}}$ for the existential type of $\bar{a}$ over parameters $\bar{c}$.

%If $\mc{M}$ is computable, the existential type of $\bar{a}$ is can be naturally represented by a c.e.~conjunction of all existential formulas true of $\bar{a}$ over $\bar{c}$. We denote the $n'th$ conjunct of this formula by $\exists$-$tp^{\bar{a}/\bar{c}}_s(\bar{x})$, and then the infinitary  $\Pi^c_2$-formula $$ \bigdoublewedge_{n \in \mathbb{N}} \mbox{ $\exists$-$tp^{\bar{a}/\bar{c}}_n (\bar{x})$} $$
%completely describes and determines $\exists$-$tp^{\bar{a}/\bar{c}}$. % $\exists$-$tp^{\bar{a}/\bar{c}}$.

%The definition says that, up to $\cl$, we can not separate existential types of independent tuples of the same arity by an existential formula.
%
%\begin{defn}\label{defn:loc} We say that independent tuples in $\mc{M}$ are \emph{locally indistinguishable} if for every $\bar{c}$ and $\bar{u}$, $\bar{v}$ independent tuples over $\bar{c}$,  for each $n \in \mathbb{N}$ there exists a $\bar{w}$ in $\cl(\bar{v}, \bar{c})$ such that:
%\begin{enumerate}
%\item  $\bar{w}$ is independent over $\bar{c}$ and $\cl(\bar{v}, \bar{c}) = \cl(\bar{w}, \bar{c})$, 
%\item $\bar{w}$ satisfies $\exists$-$tp^{\bar{v}/\bar{c}}_n$.

%\end{enumerate}

%\end{defn}

The definition below clarifies the informal discussion from the introduction. 
We say that a tuple $\bar{x}$ is independent if the set of its components is independent. Independence over a tuple is defined similarly.

\begin{defn}\label{defn:loc} We say that independent tuples in $\mc{M}$ are \emph{locally indistinguishable} if for every tuple $\bar{c}$ in $\mc{M}$ and $\bar{u}$, $\bar{v}$ independent tuples over $\bar{c}$,  
for each existential formula $\phi$ such that $\mc{M} \models \phi(\bar{c}, \bar{u})$, there exists a tuple $\bar{w}$ that is independent over $\bar{c}$, has $\mc{M} \models \phi (\bar{c}, \bar{w})$, and (with $\bar{w} = (w_1,\ldots,w_n)$ and $\bar{v} = (v_1,\ldots,v_n)$) we have $w_i \in \cl (\bar{c},v_1,\ldots,v_i)$ for $i = 1,\ldots,n$.\footnote{In general, we will use $w_i$ to denote the entries of a tuple $\bar{w}$ in this way.}
\end{defn}

Thus, even if $\bar{u}$ and $\bar{v}$ can be separated by an $\exists$-formula, we can always find independent witnesses for any $\exists$-formula true of $\bar{u}$ within the $\cl$-span of $\bar{v}$. The conditions $\bar{w}$ is independent over $\bar{c}$ and $w_i \in \cl (\bar{c},v_1,\ldots,v_i)$ are equivalent to $w_i$ and $v_i$ being interdependent over $\bar{c},v_1,\ldots,v_{i-1}$ (and hence also over $\bar{c},w_1,\ldots,w_{i-1}$).

\section{Proof of the metatheorem}
\label{sec:main-thm}

\def\G{\mc{G}}
\def\M{\mc{M}}
\newcounter{contenumi}
\def\om{\omega}
\def\bbar{\bar{b}}
\def\cbar{\bar{c}}
\def\ubar{\bar{u}}
\def\vbar{\bar{v}}
\def\xbar{\bar{x}}
\def\ybar{\bar{y}}
\def\Si{\Sigma}
\def\isom{\cong}

% REMEMBER TO COMMENT THIS OUT. ARE THEY USED ELSEWHERE?
%\usepackage{enumitem}% http://ctan.org/pkg/enumitem
%\setlist[enumerate]{noitemsep, topsep=0pt}

We fix a computable $\mathcal{M}$ and a r.i.c.e.~pregeometry $\cl$ on $\mc{M}$.
This section is devoted to a proof of Theorem~\ref{main}.
 We first prove that  Condition G implies that $\mathcal{M}$ has a computable copy with a computable basis, and then we show that Condition~B guarantees the existence of a copy with no computable basis.

 \subsection{A computable copy with a computable basis}

Recall that the \emph{independence diagram} $\mc{I}_{\mc{M}}(\bar{c})$ of $\bar{c}$ in $\mc{M}$ is the collection of all existential first-order formulas $\exists \bar{y} \psi(\bar{c}, \bar{y}, \bar{x})$ such that $\mc{M} \models \exists \bar{y} \psi(\bar{c}, \bar{y}, \bar{a})$ for some tuple $\bar{a} \in \mc{M}$ independent  over $\bar{c}$. By our assumption, $(\mathcal{M}, \cl)$ satisfies:
 
\medskip

\noindent  \textbf{Condition G:} Independent tuples are locally indistinguishable in $\mc{M}$ and for each $\mc{M}$-tuple $\bar{c}$, $\mc{I}_{\mc{M}}({\bar{c}})$ is computably enumerable uniformly in~$\bar{c}$.

\medskip

 %by Proposition~\ref{comp-pregeomtry}.).  

 \begin{prop}\label{pr:G}\nonitalics Suppose that $\mc{M}$ is a computable structure, and let $\cl$ be a r.i.c.e.~pregeometry on $\mc{M}$. If $(\mc{M}, \cl)$ satisfies Condition G, then there exists a computable copy $\mc{G}$ of $\mc{M}$ having a computable basis. Furthermore, $\mc{G} \cong_{\Delta^0_2} \mc{M}$.
 \end{prop}

\medskip

\begin{proof}
 It is sufficient to prove the proposition in the case where the dimension of $\mc{M}$ is infinite. We may also assume that the language is relational as any $\Sigma_1$ formula involving function symbols can be replaced by a $\Sigma_1$ formula which involves only the relation symbols for the graphs of those functions.

At stage $s$ of the construction we will define three things: a finite one-to-one map $\tau_s$ from an initial segment of $\om$ to $\M$, a finite tuple $a_0, \ldots, a_s$ inside the domain of $\tau_s$ and a number $t_s>s$.
We will also define a finite structure $\G_s$ by pulling back the structure $\M$ through $\tau_s$ (and, if the language is infinite, at stage $s$ we consider only the first $s$ many relations in the language). 
This structure will never change, and thus we will have $\G_s\subseteq\G_{s+1}$ and at the end we will get that $\G=\bigcup_s\G_s$ is a computable structure.
The elements $a_i$ will never change either, and we will make sure they end up forming a computable basis for $\G$.
The partial isomorphisms $\tau_s$ will change from stage to stage, but they will stabilize pointwise and hence they will have a $\Delta^0_2$ limit, $\tau$, which will be an isomorphism $\tau\colon \G\to\M$. 
The number $t_s$ represents how much we have looked into $\M$ to guess which elements are independent and which are not.
Of course we would like to be able to pick $\tau_s(a_0), \ldots, \tau_s(a_s)$ so that they are independent in $\M$ from the beginning, but since this is a $\Pi^0_1$ property, we will not be always correct.
For starters, we choose $\tau_s(a_0), \ldots, \tau_s(a_s)$ to be $t_s$-independent (see Definition \ref{def:clk}).
But to make sure we can recover from our mistakes we will need to ask for some extra assurances.
At each stage $s$, $\tau_s$, $a_0, \ldots, a_s$, and $t_s$ will satisfy properties (\ref{P1})-(\ref{P6}) which we describe below.
We start with the most obvious ones.

\subsubsection*{Basic properties} 

Let us start with independence:

\begin{enumerate}\renewcommand{\theenumi}{P\arabic{enumi}}
\item $\tau_s(a_0), \ldots, \tau_s(a_s)$ are $t_s$-independent.               \label{P1}
\end{enumerate}

Once we show that the values of $\tau_s(a_i)$ eventually stabilize, property (\ref{P1}) guarantees that their limits $\tau(a_0), \tau(a_1),....$ are independent in $\M$.
Therefore, $a_0, a_1,...$ will end up being independent in $\G$.

The second property is about the range of $\tau_s$.

\setcounter{contenumi}{\value{enumi}} \begin{enumerate} \setcounter{enumi}{\value{contenumi}}\renewcommand{\theenumi}{P\arabic{enumi}}
\item $\{0,..,s-1\}\subseteq range (\tau_s) \subseteq \cl_{t_s}(\tau_s(a_0), \ldots, \tau_s(a_s))$.    \label{P2}
\end{enumerate}

As one might expect, property (\ref{P2}) will be useful to show that $\tau$ is onto $\M$.
Also, the right inclusion will help us to keep things tight when we want to prove that the $\tau_s$'s stabilize pointwise.

\subsubsection*{Compatibility}

Usually, in this kind of construction, one requires that the atomic facts true in $\G_{s-1}$ remain unchanged in $\G_s$ and unchanged for the rest of the construction.
In our case, we also want that whenever some tuple is $t_{s-1}$-dependent, it will stay dependent forever.
Since $t_{s-1}$-dependence is witnessed by some $\Si_1$ formula being true in $\M_{t_{s-1}}$, we will ask for all such formulas to be preserved.

\setcounter{contenumi}{\value{enumi}} \begin{enumerate} \setcounter{enumi}{\value{contenumi}}\renewcommand{\theenumi}{P\arabic{enumi}}
\item For every tuple $\bbar$ in the domain of $\tau_{s-1}$ and every $\Si_1$ formula $\varphi(\xbar)$ (using only the first $t_{s-1}$ relation symbols) we have that    \label{P3}
\[
\M_{t_{s-1}}\models \varphi(\tau_{s-1}(\bbar))
	\quad \Rightarrow\quad
\M_{t_{s}}\models \varphi(\tau_{s}(\bbar)).
\]
\end{enumerate}

In particular, it follows from this that if $\tau_{s-1}(b)\in \cl_{t_{s-1}}(\tau_{s-1}(a_0),...,\tau_{s-1}(a_i))$, then $\tau_{s}(b)\in \cl_{t_{s}}(\tau_{s}(a_0),...,\tau_{s}(a_i))$.

\subsubsection*{Stabilization}
To make sure that the sequence of $\tau_s$ converges pointwise, we will require that once $\tau_s(a_0), \ldots, \tau_s(a_i)$ are really independent, they never change again.
Also, nobody that appears to be in their closure will change either. 
More formally: for $s>0$,

\setcounter{contenumi}{\value{enumi}} \begin{enumerate} \setcounter{enumi}{\value{contenumi}}\renewcommand{\theenumi}{P\arabic{enumi}}
\item Suppose that $b\in \dom(\tau_{s-1})$ and that $\tau_{s-1}(b) \in \cl_{t_{s-1}}(\tau_{s-1}(a_0), \ldots, \tau_{s-1}(a_i))$.
Then, we can only have $\tau_s(b)\neq \tau_{s-1}(b)$, if $\tau_{s-1}(a_0), \ldots, \tau_{s-1}(a_i)$ are $t_s$-dependent.     \label{P4}
\end{enumerate}

In particular, if $\tau_{s-1}(a_0), \ldots, \tau_{s-1}(a_i)$ stay $t_s$-independent at stage $s$, then their values won't change at $s$.
And if $\tau_{s-1}(a_0), \ldots, \tau_{s-1}(a_i)$ are actually independent, their values will never change.
Therefore, once we show that $\tau_{s}(a_0), \ldots, \tau_{s}(a_i)$ are eventually going to be independent, we get that their values eventually stabilize. 

Furthermore, by condition (\ref{P2}), we get that for every $b$ in the domain of $\tau_s$, $\tau_s(b)\in \cl_{t_{s}}(\tau_{s}(a_0), \ldots, \tau_{s}(a_s))$.
By condition (\ref{P3}) this fact will never change at future stages.
So, once $\tau_{r}(a_0), \ldots, \tau_{r}(a_s)$ become independent at some future stage $r>s$, condition (\ref{P4}) implies that $\tau_r(b)$ will never change again.

\subsubsection*{Safeness}

The idea of condition (\ref{P5}) below is that before choosing values for $\tau_s(a_0),...,\tau_s(a_s)$ we will not only ask for them to be $t_s$-independent, but also that whatever we have committed  about them is consistent with the corresponding independence diagrams, as if they were actually independent.
This way, if we later realize they were not independent, we can find some other (potentially) independent tuple which is compatible with our construction thus far.

We start by defining the formulas that describe the commitments we have made so far.
Suppose we have $(\tau_s; a_0, \ldots, a_s)$ as above.
Fix $i<s$.
For each $j$, let $u_j=\tau_s(a_j)$.
Let $\cbar$ be the tuple of elements in the range of $\tau_s$ which belong to $\cl_{t_s}(u_0,...,u_i)$.
(These are the elements that, by condition (\ref{P4}), we do not want to move if $u_0,...,u_i$ were to be independent.)
Let $\vbar$ be the elements of $\M_{t_s}$ which are neither in $\cbar$ nor in $\ubar$, that is, those elements which are not in $\bar{u}$, and which are not in the $\cl_{t_s}$-closure of $u_0,\ldots,u_i$, but (by \ref{P2}) are in the closure of $\bar{u} = (u_0,\ldots,u_s)$.
So $\M_{t_s}=\{\cbar\}\cup \{\vbar\}\cup\{u_{i+1},...,u_s\}$.
Let $\theta(\cbar, \vbar, u_{i+1},...,u_s)$ be the conjunction of the atomic diagram of $\M_{t_s}$ (using only the first $t_s$ relation symbols).
Define
\[
\psi_i(\cbar,x_{i+1},...,x_s) \leftrightharpoons \exists \ybar\theta(\cbar, \ybar,x_{i+1},...,x_s). 
\]
Note that $\psi_i$ is an $\exists$-formula with parameters $\cbar$, and with $s-i$ indeterminates $x_{i+1},...,x_s$.

\setcounter{contenumi}{\value{enumi}} \begin{enumerate} \setcounter{enumi}{\value{contenumi}}\renewcommand{\theenumi}{P\arabic{enumi}}
\item For each $i<s$, $\psi_i(\cbar,\xbar)$ belongs to the independence diagram of $\cbar$.    \label{P5}
\end{enumerate}

Recall that Condition G tells us that these independence diagrams are uniformly c.e., and hence we can always wait to see if a formula shows up in one.
Thus condition (\ref{P5}) is $\Si^0_1$.
All the previous conditions were computable.

Also, let us remark that if $u_0,...,u_s$ are independent in $\M$, then these formulas do belong to the corresponding independence diagrams, and hence (\ref{P5}) holds.

\subsubsection*{Least Span}
We need this last property to guarantee that, as we are moving the values of $\tau_s(a_0),...,\tau_s(a_i)$ around, we will eventually fall on one that is actually independent.
Also, it will help us ensure that they form a basis at the end.

If we had a $0'$ oracle, we could build what we call {\em the $\om$-least basis} for $\M$ as follows:
once we have defined $m_0,m_1,....,m_{s-1}$, we define the next element, $m_s$, be the $\omega$-least element of $\M$ which is not in the closure of $m_0,m_1,....,m_{s-1}$. (Recall that the domain of $\mc{M}$ is $\omega$.)
We will not be able to get $\tau(a_0), \tau(a_1),...$ to be this particular basis, but we can get close.
We say that a tuple $u_0,...,u_s$ in $\M$ has {\em $\om$-least span} if, for every $i \leq s$, $\cl(u_0,...,u_i) = \cl(m_0,...,m_i)$, or equivalently, if  $\cl(u_0,...,u_i)$ contains the $\om$-least element of $\M$ which is not in $\cl(u_0,...,u_{i-1})$. By the exchange principle and the fact that $m_0,m_1,\ldots$ are independent, $m_i \in \cl(u_0,\ldots,u_i)$ and $u_i \in \cl(u_0,\ldots,u_{i-1},m_i)$.
Note that if an infinite subset of $\M$ has $\om$-least span then it is a basis.

At each stage $t$ we can only approximate this property. At stage $t$, let $n_0[t]$ be the least element not in $\cl_t(\varnothing)$, $n_1[t]$ the least element not in $\cl_t(n_0[t])$, $n_2[t]$ the least element not in $\cl_t(n_0[t],n_1[t])$, and so on. We say that a tuple $u_0,...,u_s$ in $\M$ has {\em $\om$-least span at $t$} if  for every $i \leq s$, $\cl_t(u_0,...,u_i)$ contains $n_i[t]$ and $\cl_t(u_0,\ldots,u_{i-1},n_i[t])$ contains $u_i$.

\setcounter{contenumi}{\value{enumi}} \begin{enumerate} \setcounter{enumi}{\value{contenumi}}\renewcommand{\theenumi}{P\arabic{enumi}}
\item $\tau_s(a_0), \ldots, \tau_s(a_s)$ has {\em $\om$-least span at $t_s$}    \label{P6}
\end{enumerate}

Note that the $n_i[t]$ are computable and are increasing (in the lexicographic order) as $t$ increases, with limit $m_i$ (the elements of the $\omega$-least basis). Suppose that at stage $t$ we have that $n_0[t],...,n_i[t]$ are part of the $\om$-least basis, i.e., they are equal to $m_0,\ldots,m_i$. If $u_0,\ldots,u_i$ has $\omega$-least span at stage $t$, then the exchange principle implies that $u_0,\ldots,u_i$ are independent.

On the other hand, if $u_0,\ldots,u_i$ are independent and have $\omega$-least span at stage $t$, then since $u_0,\ldots,u_i \in \cl_t(n_0[t],\ldots,n_i[t])$, $n_0[t],\ldots,n_i[t]$ are independent and hence equal to $m_0,\ldots,m_i$. Then $u_0,\ldots,u_i$ have $\omega$-least span at stage $t$.

Therefore, we get that, for every $i\in\om$, there is $t$ large enough such that, once $t_s>t$, Condition (\ref{P6}) implies that $\tau_s(a_0),...,\tau_s{(a_i)}$ are independent in $\M$. As we have argued above, this then implies that the sequence $\tau_s$ converges pointwise to a bijection $\tau\colon \om \to\M$ getting that $\G\isom \M$. Condition (\ref{P6}) also implies that $\tau(a_0), \tau(a_1),...$ is a basis of $\M$ and hence that $a_0, a_1,...$ is a basis of $\G$.

\subsubsection*{Construction} 
At \emph{stage $0$},  define $\mc{G}_0 = \{a_0\}$ and let  $\tau_0(a_0) $ be the $\omega$-least element of $\mc{M}$ that is independent over~$\varnothing$. Without loss of generality, we assume $\tau_0(a_0) = 0$.
After \emph{stage $s$} is finished, we will have defined $\tau_{s}$, $a_0, \ldots, a_s \in \mc{G}_{t_s}$ and $t_s$ satisfying (\ref{P1})-(\ref{P6}). 
At \emph{stage $s+1$} define  $\tau_{s+1}$, $a_0, \ldots, a_{s+1} \in \mc{G}_{t_{s+1}}$ and $t_{s+1}$ to be the first ones we find satisfying (\ref{P1})-(\ref{P6}). 

The last step in the proof is to show that Condition G guarantees such objects exist. 

\begin{claim}
Given $\tau_{s}$, $a_0, \ldots, a_s$ and $t_s$ satisfying (\ref{P1})-(\ref{P6}), there exist $\tau_{s+1}$, $a_0, \ldots, a_{s+1}$ and $t_{s+1}$ also satisfying (\ref{P1})-(\ref{P6}).
\end{claim}
\begin{proof}
To simplify the notation let $u_i=\tau_s(a_i)$ for each $i$.

Suppose now that $u_0,...,u_s$ are not independent.
Let $i$ be as large as possible with $\tau_s(a_0),...,\tau_s(a_i)$ are independent.
We noted above that this implies that they also have $\om$-least span.
Let $\cbar$ be the tuple of elements in the range of $\tau_s$ which belong to $\cl_{t_s}(u_0,...,u_i)$.
We will keep $\tau_s$ fixed on (the pre-image of) $\cbar$ so that condition (\ref{P4}) is satisfied, and change the rest.
Let $\psi_i(\cbar,x_{i+1},...,x_s)$ be as in the subsection on Safeness.
Since $\psi_i$ is in the independence diagram of $\cbar$ we know it is true of some independent tuple.
We are now ready to apply Condition G.
Let $v_{i+1},...,{v_s}$ be independent over $\cbar$ such that $u_0,...,u_i,v_{i+1},..,v_s$ has $\om$-least span.
By Condition $G$ there exist $w_{i+1},...,{w_s}$ independent over $\cbar$ such that $\M\models \psi_i(\cbar, w_{i+1},...,w_s)$ and such that for every $j\leq s$, $w_j\in \cl(\cbar,v_{i+1},...,v_{j})$ (and, by the exchange principle, $\cl(\cbar,v_{i+1},\ldots,v_j) = \cl(\cbar,w_{i+1},\ldots,w_j)$ for each $i < j \leq s$).
Thus $u_0,\ldots,u_i,w_{i+1},\ldots,w_s$ has $\om$-least span as well.

Re-define $\tau_{s+1}(a_j)=w_j$ for $j>i$.
For all the other elements in the domain of $\tau_s$ use the fact that $\M\models \psi_i(\cbar, w_{i+1},...,w_s)$ to define $\tau_{s+1}$ in a way their existential diagrams within $\M_{t_s}$ remain unchanged, so that (\ref{P3})  holds.
Let $t_{s+1}$ be larger than all of the elements in the image of $\tau_{s+1}$.

Since $u_0,...,u_i,w_{i+1},...,w_{s}$ is independent we get conditions (\ref{P1}) and (\ref{P5}).
By increasing $t_{s+1}$, we can have $n_j[t_{s+1}] = m_j$ for $0 \leq j \leq s$.
For every $i < j \leq s$, $m_j \in \cl(\cbar,w_{i+1},...,w_{j})$ and $w_j \in \cl(\cbar,w_{i+1},...,w_{j-1},m_j)$.
Increase $t_{s+1}$ further to witness these dependencies.
Thus $u_0,...,u_i,w_{i+1},..,w_s$ has $\om$-least span at stage $t_s$, and we have condition (\ref{P6}).

If $u_0,...u_s$ were independent from the beginning, leave everything untouched.
In either case we have $\tau_{s+1}, a_0,...,a_s, t_{s+1}$ satisfying all the required condition maybe except (\ref{P2}).
Define $\tau_{s+1}(a_{s+1})$ to be the least element not in $\cl(u_0,...,u_s)$.
Increase the value of $t_{s+1}$ if necessary so that every element below $\tau_{s+1}(a_{s+1})$ belongs to $\cl_{t_{s+1}}(u_0,...,u_s)$.
This way we get that $u_0,...,u_s,\tau_{s+1}(a_{s+1})$ is $\om$-least at $t_{s+1}$.
If necessary, extend $\tau_{s+1}$ so that $s+1$ belongs to its range.
Since $u_0,...,u_s$ had $\om$-least span, $s+1$ must be in the $t_{s+1}$-span of $u_0,...,u_{s+1}$.
It is not hard to check now that (\ref{P1})-(\ref{P6}) are all satisfied.
\end{proof}

This finishes the proof of Proposition \ref{pr:G}.
\end{proof}

\begin{rem} Note that every  computable copy of $\mc{M}$ with a computable basis satisfies Condition G trivially. Thus, having some computable copy with Condition G is necessary and sufficient for $\mc{M}$ to have a computable copy with a computable basis (though it is possibly not necessary for $\mc{M}$ itself to have Condition G).
\end{rem}

\subsection{A computable copy with no computable basis} Recall that \emph{dependent elements are dense in $\mc{M}$} if, whenever $\psi$ is quantifier-free and $\mc{M} \models \exists \bar{y} \psi(\bar{c}, \bar{y}, a)$ for a non-empty tuple $\bar{c}$ and $a \in \mc{M}$, there is $b \in \cl(\bar{c})$ and $\mc{M} \models \exists \bar{y} \psi(\bar{c}, \bar{y}, b)$. 
We now prove:
\begin{prop}\label{pr:B}\nonitalics Let $\mc{M}$ be a computable structure, and let $\cl$ be a r.i.c.e.~pregeometry upon $\mc{M}$. If the $\cl$-dimension of $\mc{M}$  is infinite and 
 dependent elements   are dense in $\mc{M}$ (this is Condition B),  then $\mc{M}$ has a computable presentation $\mc{B} \cong_{\Delta^0_2} \mc{M}$ that has no computable basis. 
\end{prop}

\begin{proof} Because the proof is a standard finite injury construction we will only give a sketch.
 Suppose $\mc{M}$ has a computable basis $a_0, a_1, \ldots$ (otherwise there is nothing to prove). We construct $\mc{B}$ by stages.
We meet the requirements:
$$R_e: \varphi_e\mbox{ is not a dependence algorithm for } \mc{B},$$
where $\varphi_0, \varphi_1, \ldots, $is the standard effective listing of all partial recursive functions. 
Initially we will attempt to copy $\mc{M}$ so that the images of $a_i$ are special elements $b_i$ of $\mc{B}$. 

The strategy for $R_e$ will have a witness $b_e$; it waits until $\varphi_e$ declares $b_e$ independent of $b_0, \ldots, b_{e-1}$. If it ever does, 
we make $b_{e}$ dependent on $b_0, \ldots, b_{e-1}$ using that dependent elements are dense in $\mc{M}$. We also rearrange the map from $\mc{M}$ to $\mc{B}$; to do that we introduce a new image for $a_e$ in $\mc{B}$.

Now this does not immediately prevent $\mc{B}$ from having a computable basis, but it does ensure that the closure operation on $\mc{B}$ is not computable. By Proposition \ref{comp-pregeometry}, this is actually equivalent. \end{proof}

Theorem~\ref{main} follows from Propositions~\ref{pr:G} and \ref{pr:B}

%%%%%%%%%%%%%%%%%%%%%%%%%%%%%%%%%%%%%%%%%%%%%%%%%%%%%%%%%%%%%%%%%%%%%%%%%%%%%%%%%%%%%%%%%%%%%%%%%%%%%%%%%%%%%%%%%%%%%%%%%%%%%%%%%%%%%%%%%%%%%%%%%%%%%%%%%%%%%%%%%%%%%%%%%%%%%%%%%%%%%%%%%%%%%%%%%%%%%%%%%%%%%%%%%%%%%%%%%%%%%%%%%%%%%%%%%%%%%%%%%%%%%%%%%%%%%%%%%%%%%%%%%%%%%%%%%%%%%%%%%%%%%%%%%%%%%%%%%%%%%%%%%%%%%%%%%%%%%%%%%%%%%%%%%%%%%%%%%%%%%%%%%%%%%%%%%%%%%%%%%%%%%%%%%%%%%%%%%%%%%%%%%%%%%%%%%%%%%%%%%%%%%%%%%%%%%%%%%%%%%%%%%%%%%%%%%%%%%%%%%%%%%%%%%%%%%%%%%%%%%%%%%%%%%%%%%%%%%%%%%%%%%%%%%%%%%%%%%%%%%%%%%%%%%%%%%%%%%%%%%%%%%%%%%%%%%%%%%%%%%%%%%%%%%%%%%%%%%%%%%%%%%%%%%%%%%%%%%%%%
\section{Applications}

It is not difficult to see that computable vector spaces and algebraically closed fields both have the  Mal{\textquotesingle}cev property and indeed satisfy Conditions B and G (essentially  Mal{\textquotesingle}cev \cite{Malcev62}, Metakide-Nerode \cite{MetakidesNerode79}, and Goncharov \cite{Goncharov82}). Since we do not really need the metatheorem for these trivial cases, we skip these examples and leave the (elementary) verification to the reader. We concentrate on non-elementary applications.

\subsection{Differentially Closed Fields}
\label{DCF}Recall that a differential field is a field with a differential operator $\delta$. 
 In this section we look at existentially closed differential fields of characteristic zero. The first-order theory $\dcf$ of differentially closed fields is complete, axiomatizable, decidable, and has quantifier elimination. It is the model companion of differential fields. A more detailed exposition can be found \cite{Marker06}. 

In differential fields, there are analogs of polynomial rings and algebraic independence. The ring $K\{X_1,\ldots,X_n\}$ of $\delta$-polynomials over $K$ is the polynomial ring
\[ K[X_1,\ldots,X_n,\delta(X_1),\ldots,\delta(X_n),\delta^2(X_1),\ldots] \]
where each of these are indeterminates. We extend the derivation to this ring in the natural way, by setting $\delta(\delta^n(X_i)) = \delta^{n+1}(X_i)$.

\begin{defn}
We say that a set $A$ is \textit{$\delta$-dependent} over $B$ if there are $a_1,\ldots,a_n$ in $A$ and a nonzero $\delta$-polynomial $f \in \mathbb{Q}\langle B \rangle\{X_1,\ldots,X_n\}$ which they satisfy. Otherwise, we say that $A$ is \textit{$\delta$-independent} over $B$. A \textit{$\delta$-transcendence base} is a maximal $\delta$-independent subset of $K$.
\end{defn}

%\begin{defn}
%We say that a field $K$ is \textit{differentially closed} if whenever $f, g \in K\{X\}$ are differential polynomials in a single variable, $g$ is nonzero, and the order of $f$ is greater than the order of $g$, then there is $a \in K$ such that $f(a) = 0$ and $g(a) \neq 0$.
%\end{defn}

We get a finitary closure operator by defining $\cl(B)$ to be the set of elements which are $\delta$-dependent\footnote{Warning: In a differentially closed field, $\delta$-dependence is always different from model-theoretic algebraic dependence as the equation $\delta(x) = 0$ defines an infinite set.} over $B$. It is not hard to verify that differential closure induces a r.i.c.e.\ pregeometry. We prove:

\medskip

\noindent \textbf{Theorem~\ref{th:diff}.} \textit{The class $DCF_0$ of computable differential closed fields  of characteristic zero has the  Mal{\textquotesingle}cev property with respect to $\delta$-independence.}

\medskip

\begin{proof}
By Theorem \ref{main}, it suffices to check that any $\mc{M} \models DCF_0$ of infinite dimension satisfies Conditions G and B.

Differentially closed fields have unique independent types. That is, for any fixed tuple $\bar{c}$, and $\bar{a}$, $\bar{b}$ both independent over $\bar{c}$, $\tp(\bar{a}/\bar{c}) = \tp(\bar{b}/\bar{c})$. Let $p(\bar{x})$ be this type with parameters $\bar{c}$. This type $p(\bar{x})$ is generated over the theory $DCF_0$ by all of the quantifier-free formulas true about $\bar{c}$ together with the formulas which say that $\bar{x}$ satisfies no non-trivial differential polynomial with coefficients from $\mathbb{Q}\{\bar{c}\}$ \cite[see Section 3 of][]{Marker02}.
%http://library.msri.org/books/Book39/files/dcf.pdf}
Since $\mc{M}$ is computable, we can list all valid quantifier-free formulas true of $\bar{c}$, and thus $p(\bar{x})$ is computable uniformly in $\bar{c}$.

Now we will check Condition G. Given $\bar{c}$, we can compute the type $p(\bar{x})$ of an independent tuple of some arity over $\bar{c}$. Then we can decide whether or not any existential formula is in this type.
Similarly, independent tuples are locally indistinguishable, because if $\bar{c}$ is a tuple, and $\bar{u}$ and $\bar{v}$ are both independent over $\bar{c}$, then $\tp(\bar{u} / \bar{c}) = \tp(\bar{v} / \bar{c})$.

We check Condition B. We first observe that  independent types are non-principal in $DCF_0$. 
Indeed, it is well-known that if $K$ is a differential field and $L$ a differentially closed field containing $K$, then the differential closure of $K$ in $L$ is a differentially closed field which omits the type of a $\delta$-transcendental element over $K$.
Now suppose $\mc{M} \models \exists \bar{y} \psi(\bar{c}, \bar{y}, a)$ for a non-empty tuple $\bar{c}$ and $a \in \mc{M}$. If  $a\in \cl(\bar{c})$ then we are done. Otherwise, consider an embedding of the prime model $\mc{K}$ over $\bar{c}$ in $\mc{M}$ (since $DCF_0$ is $\omega$-stable, it has prime models over any set). Since $\mc{K} \preceq \mc{M}$,    $\mc{K} \models \exists b \exists \bar{y} \psi(\bar{c}, \bar{y}, b)$. Any such $b$ will be differentially algebraic over $\bar{c}$.
\end{proof}

% It remains to apply Theorem~\ref{main}. 

\subsection{Difference Closed Fields}\label{sec:ACFA}
 A \textit{difference field} is a field together with a distinguished automorphism $\sigma$. Difference fields have a model companion $\acfa$. The theory $\acfa$ of \textit{difference closed fields} is first-order axiomatizable.
 The theories $\acfa$, $\acfa_0$, and $\acfa_p$ (the subscript denoting the characteristic) are decidable, see \cite[(1.4) of][]{ChatzidakisHrushovski99}. Note that the field-theoretic and model-theoretic algebraic closures coincide, see \cite[(1.7) of][]{ChatzidakisHrushovski99}.
 For a more detailed exposition of the model theory of difference fields, see \cite{ChatzidakisHrushovski99}.

  The difference polynomial ring $K\langle X_1,\ldots,X_n \rangle$ is the polynomial ring
\[ K[X_1,\ldots,X_n,\sigma(X_1),\ldots,\sigma(X_n),\sigma^2(X_1),\ldots] \]
with the natural extension of $\sigma$.

\begin{defn}
Let $A$ be a subset of $K$, and let $E$ be the difference field generated by $A$. We say that $a_1,\ldots,a_n$ are \textit{transformally dependent} over $A$ if there is a nontrivial difference polynomial $f \in E\langle X \rangle$ which they satisfy. Otherwise, we say that $a_1,\ldots,a_n$ are \textit{transformally independent} over $A$.
\end{defn}

We get a finitary closure operator by defining $\cl(B)$ to be the set of elements which are transformally dependent over $B$. This notion of dependence induces a r.i.c.e.\ pregeometry. We prove:

\medskip

\noindent \textbf{Theorem~\ref{th:acfa}.} \textit{The class $\acfa$ of computable difference closed fields has the  Mal{\textquotesingle}cev property with respect to transformal independence.}

\medskip

\begin{proof} 
By Theorem \ref{main}, it suffices to check that any $\mc{M} \models \acfa$ of infinite dimension satisfies Conditions G and B.

Let $\mc{M}$ be a computable model of $\acfa$. The complete theory of $\mc{M}$ is given by the axioms of $\acfa$, the characteristic, and the action of the automorphism on the algebraic closure of the prime field. Every formula is equivalent, modulo $\acfa$, to an existential formula \cite[see (1.6) of][]{ChatzidakisHrushovski99}. So the elementary diagram of $\mc{M}$ is decidable, since for any formula $\varphi$, we will eventually find that either $\mc{M} \models \varphi$ or $\mc{M} \models \neg \varphi$.

We check Condition G.
Difference closed fields have unique independent types over \emph{algebraically closed} subfields (in fact, as we will see, over any parameters, but the types over arbitrary parameters may be more complicated). Thus, for any algebraically closed $E \subseteq \mc{M}$ and $\bar{a}$, $\bar{b} \in \mc{M}$ both independent over $E$, $$\tp(\bar{a}/E) = \tp(\bar{b}/E).$$ Moreover, this type $p(\bar{x})$ is generated by all of the quantifier-free formulas true about $E$ together with the formulas which say that $\bar{x}$ satisfies no non-trivial difference polynomial with coefficients from $\mathbb{Q}\{E\}$ \cite[see Proposition 2.10 of][]{ChatzidakisHrushovski99}.
Note that if $\bar{c}$ is any tuple, then there is a unique type of an independent element over $\acl(\bar{c})$, and hence over $\bar{c}$.
Thus, independent tuples are locally indistinguishable.
% However, the type is not as easily described, since the type includes facts about the action of the automorphism on the algebraic closure of~$\bar{c}$.
 Given $\bar{c}$, we can enumerate $\acl(\bar{c})$ and compute the type $p(\bar{x})$ of an independent tuple over $\acl(\bar{c})$. We can restrict ourselves  to formulas about $\bar{c}$ and compute the type of an independent tuple over $\bar{c}$. Then we can decide whether or not any existential formula is in this type.
%Independent tuples are locally indistinguishable: as we noted above, if $\bar{c}$ is a tuple, and $\bar{u}$ and $\bar{v}$ are both independent over $\bar{c}$, then $\tp(\bar{u} / \bar{c}) = \tp(\bar{v} / \bar{c})$.

We check Condition B. As in the case of $DCF_0$,  independent types are non-principal in $ACFA$. This can be derived from the fact that if $E$ is a difference field contained in a difference closed field $K$, then the set $K_0$ of elements transformally algebraic over $E$ is a difference closed field with $K_0 \preceq K$ \cite[see, e.g., the remark after Theorem 1.1 of][]{ChatzidakisHrushovski99}. If $\bar{e}$ is a tuple in some difference closed field $K$, and $E$ is the difference field generated by $\bar{e}$, then the corresponding $K_0 \preceq K$ omits the type of an independent tuple over $\bar{e}$. The rest can be done just as in the case of $DCF_0$. \end{proof}

\subsection{Real Closed Fields}
\label{RCF}
%We begin by defining real closed fields and sumarising some model-theoretic results about them and about o-minimal structures.
We assume that the reader is familiar with the definition of an ordered field. \emph{Real closed fields} are existentially closed ordered fields. Equivalently, a  field $F$ is real closed if every positive element has a square root in $F$ and every polynomial of odd degree with coefficients in $F$ has a root in $F$.
These give axioms for the theory $\rcf$ of real closed fields. Tarski~\cite{Tarski48} showed that the theory $\rcf$ is complete, decidable, and has quantifier elimination. Using quantifier elimination, it is easy to see that the definable sets in a single variable consist of finitely many points (the solutions to certain polynomial equations) plus finitely many (possibly unbounded) intervals. The standard generalization of this phenomenon is \emph{o-minimality}, see e.g.~\cite{VanDenDries98}.  %(where o stands for ``order'').

%\begin{defn}
%An \textit{ordered field} is a field $F$ with a linear ordering $<$ which satisfies, for all $a$, $b$, and $c$:
%\begin{enumerate}
%	\item if $a \leq b$ then $a + c \leq b + c$, and
%	\item if $a \leq b$ and $c \geq 0$, then $a c \leq b c$.
%\end{enumerate}
%\end{defn}

%Real closed fields are ordered fields which contain as many solutions to polynomials as possible -- a field in which a negative number is a square cannot be ordered.

%\begin{defn}
%A real closed field $F$ is an ordered field such that every positive element has a square root in $F$ and every polynomial of odd degree with coefficients in $F$ has a root in $F$.
%\end{defn}

%\begin{defn}
%A structure $\mc{M} = (M,<,\ldots)$ is \textit{o-minimal} if every definable set in one dimension is a union of points and intervals.
%\end{defn}

%A good reference on o-minimal structures is \cite{VanDenDries98}. 

Let $F$ be a real closed field. The model-theoretic algebraic closure agrees with the algebraic closure as a pure field, as any finite set of points which is definable can be defined without the ordering. In general, Pillay and Steinhorn \cite{PillaySteinhorn86} remarked that for o-minimal structures, the model-theoretic algebraic closure agrees with the model-theoretic definable closure, and that this is always a pregeometry.

Unlike algebraically closed fields, transcendental types are not unique. For example, in $\mathbb{R}$, $\pi$ and $e$ have different types over $\mathbb{Q}$ and yet are both transcendental. Our proof that real closed fields have the  Mal{\textquotesingle}cev property will use \emph{cell decomposition}.

\begin{defn}
The collections of \textit{cells} is defined recursively by:
\begin{enumerate}%[label=(\arabic{*})]
	\item If $X$ is a single point in $F^n$, then $X$ is a 0-dimensional cell.
	\item Every open interval $(a,b)$ in $F$ is a 1-dimensional cell (with $a \in F \cup \{-\infty\}$ and $b \in F \cup \{\infty\})$.
	\item If $X \subseteq F^n$ is an $m$-dimensional cell, and $f : X \to F$ is a continuous definable function, then
		\[ Y = \{(\bar{x},f(\bar{x})) : \bar{x} \in X\} \]
		is an $m$-dimensional cell.
	\item If $X \subseteq F^n$ is an $m$-dimensional cell, and $f$ and $g$ are either both continuous functions $X \to F$, or $f$ is possibly the constant function $-\infty$, or $g$ is possibly the constant function $\infty$, and $f(\bar{x}) < g(\bar{x})$ for all $\bar{x} \in X$, then
	\[ Y = \{(\bar{x},y) : \bar{x} \in X  \text{ and } f(\bar{x}) < y < g(\bar{x}) \} \]
	is an $m+1$-dimensional cell.
\end{enumerate}
\end{defn}

Every definable set can be built up from cells using the following theorem:

\begin{thm}[Cell Decomposition, \cite{KnightPillay86}]\nonitalics
Given a set $X$ definable over $\bar{a}$, we can write $X$ as a finite union of disjoint cells. The functions and endpoints of the intervals are all definable over $\bar{a}$.
\end{thm}

We note that a cell decomposition is uniquely described by its rank, end points, and the definable functions used in its definition.

\medskip

\noindent \textbf{Theorem~\ref{th:rcf}.} \textit{The class of computable real closed fields has the  Mal{\textquotesingle}cev property with respect to algebraic independence.}
\medskip

 \begin{proof} We first observe that cell decompositions can be \emph{computed}.

\begin{claim}\label{kuku1}\nonitalics Let  $\mc{M}$ be a computable real closed field. There exists a uniform procedure that, given a first-order $\phi$ defining $X_\phi \subseteq \mc{M}^m$ with parameters $\bar{a} \in \mc{M}$, outputs the cell decomposition of $X_{\phi}$.
\end{claim}

\begin{proof} Since $Th(\mc{M})$ admits elimination of quantifiers, the type of $\bar{a}$ is (uniformly) computable. There exists a first-order formula stating that these points, intervals and continuous functions  correspond to a cell decomposition of a $X_{\phi}$. We can list all formulas of this form. Thus, we will eventually find the right formula. We will then extract the definitions of the functions and end-points from the formula.\end{proof}

\noindent The next claim provides us with a condition for a definable set to contain an independent tuple.

\begin{claim}\label{kuku2}\nonitalics
Let $\mc{M}$ be a real closed field of infinite transcendence degree. Let $X \subseteq \mc{M}^n$ be a definable set with parameters $\bar{a}$. Then $X$ contains a tuple algebraically independent over $\bar{a}$ if and only if the cell decomposition of $X$ contains a cell of dimension $n$.
\end{claim}
\begin{proof}
It is easy to see that any tuple in a cell of dimension strictly less than $n$ is algebraically dependent over $\bar{a}$. This is because such a cell must be built using, at some point, (1) or (3) from the definition of cell decomposition. If it uses (1), then this point is definable. If it uses (3), then any tuple $(x_1,\ldots,x_n) \in X$ has $x_i = f(x_1,\ldots,x_{i-1})$ for some $i$ and some function $f$ definable over $\bar{a}$. Then $x_i$ is definable over $x_1,\ldots,x_{i-1},\bar{a}$ and hence algebraically dependent over them.

Now suppose that the cell decomposition of $X$ contains a cell $D$ of dimension $n$. Suppose that the cell $D$ is built up from an interval $(c,d)$ using functions $$(f_2,g_2),\ldots,(f_n,g_n).$$ We will assume that the functions are bounded but the case when some of them are $\pm \infty$ requires just a simple modification. Let $b_1,\ldots,b_n$ be algebraically independent over $\bar{a}$. We may assume that each $b_i$ satisfies $0 < b_i < 1$ (by replacing $b_i$ by $-b_i$, $b_i^{-1}$, or $-b_i^{-1}$ if necessary). Let $b_1' = c + (d - c)b_1$; so $b_1' \in (c,d)$. Then, for $i=1,\ldots,n-1$, let $$b_{i+1}' = f_{i+1}(b_1',\ldots,b_i')(b_{i+1}) + g_{i+1}(b_1',\ldots,b_i')(1-b_{i+1}).$$ Note that $(b_1',\ldots,b_n')$ is in the open cell $D$. Also, $b_1'$ is interdefinable with $b_1$ over $c$ and $d$, and since $c$ and $d$ are $\bar{a}$-definable, $b_1$ and $b_1'$ are interalgebraic over $\bar{a}$. Similarly, since each $f_i$ and $g_i$ is $\bar{a}$-definable, $b_2$ and $b_2'$ are interalgebraic over $\bar{a},b_1'$. In general, $b_{i+1}$ and $b_{i+1}'$ are interalgebraic over $\bar{a},b_1',\ldots,b_i'$. Since $b_1,\ldots,b_n$ are independent over $\bar{a}$, $b_1',\ldots,b_n'$ are independent over $\bar{a}$.
\end{proof}

By Theorem \ref{main}, it suffices to check that any $\mc{M} \models RCF$ of infinite dimension satisfies Conditions G and B.

We now argue that every computable real closed $\mc{M}$ of infinite transcendence degree satisfies Condition G.
It follows from Claims~\ref{kuku1} and \ref{kuku2} above that we can effectively and uniformly list the independence diagram of each $\bar{a} \in\mc{M}$. It follows at once from the proof of Claim~\ref{kuku2}, where we had no restrictions on the choice of $b_1, \ldots, b_n$ (and hence can take them to be $u_1,\ldots,u_n$), that independent tuples are locally indistinguishable in $\mc{M}$.

We show that $\mc{M}$ satisfies Condition B. 
Let $\bar{c}$ be a tuple from $\mc{M}$ and let $\bar{a}$ be independent over $\bar{c}$. Then any formula $\varphi(\bar{c},\bar{x})$ with parameters $\bar{c}$ true of $\bar{a}$ defines, by Claim \ref{kuku2}, a $\bar{c}$-definable set of dimension $n$. Such a set contains a definable open cell whose endpoints and functions are $\bar{c}$-definable, and hence it contains some point $\bar{b} \in \mathbb{Q}(\bar{c})$. For example, if the open cell is built up from an interval $(p,q)$ using bounded functions $(f_2,g_2),\ldots,(f_n,g_n)$ then let $b_1$ be the midpoint of the interval $(p,q)$, $b_2$ the midpoint of the interval $(f_2(b_1),g_2(b_1))$, and so on. So $\bar{b} \in \cl(\bar{c})$ satisfies $\varphi(\bar{c},\bar{x})$.
\end{proof}

\subsection{Torsion-free abelian groups}\label{sec:tfag}
Recall that an abelian group is torsion-free if it has no non-zero elements of finite order. We cite Fuchs~\cite{Fuchs70,Fuchs73} for background on infinite abelian groups. 
The notion of independence is the usual linear independence, but the coefficients are taken from~$\mathbb{Z}$. We will refer to parts of the proof below when we consider ordered abelian groups in the following section.

\medskip
\noindent \textbf{Theorem~\ref{th:tfag}.}  \textit{The class of computable torsion-free abelian groups has the  Mal{\textquotesingle}cev property with respect to $\mathbb{Z}$-independence.}

\medskip

\begin{proof} 
Recall that a subgroup $H \leqq A$ of an abelian group $A$ is \emph{pure} if for any integer $m$ and each $h \in H$,  $$(\exists g \in G) \, mg = h \, \Rightarrow \,  (\exists w \in H) \, mw = h.$$ 
It is well-known (see \cite{Kaplansky69}) that a finitely generated pure subgroup of an abelian group $A$ detaches as a direct summand of $A$. 
The $\mathbb{Z}$-dimension of an abelian group $A$ is often called the \textit{rank} of $A$, but to be consistent with our notation for pregeometries we will call it the dimension of $A$, $\dim(A)$. 

Let $\mc{M}$ a computable torsion-free abelian group of infinite dimension. 

\begin{claim}\label{cl:tfa1}  $\mc{M}$ satisfies Condition B. \end{claim}
\begin{proof}[Proof of Claim]
Suppose $\mc{M} \models \exists \bar{y} \psi(\bar{c}, a, \bar{y})$, where $\psi$ is a conjunction of linear equations  and negations of linear equations. By a linear equation we mean a linear equation over $\mathbb{Z}$.
We must find an element $b$ that is $\mathbb{Z}$-dependent over $\bar{c}$ and satisfies $\exists \bar{y} \psi(\bar{c}, b, \bar{y})$.

Fix any tuple $\bar{w}$ witnessing the existential quantifier $\exists \bar{y}$. 
Consider the subgroup $\mc{X}$ of $\mc{M}$ generated (as a subgroup, rather than as a pure subgroup) by  $\bar{c}, a, \bar{w}$, and let $C$ be the least pure subgroup  of $\mc{X}$ that contains $\bar{c}$. Note that $C = \cl(\bar{c}) \cap \mc{X}$. Since $\mc{X}$ is finitely generated and torsion-free, it is free abelian and thus is isomorphic to a direct sum of finitely many copies of $\mathbb{Z}$ (\cite{Fuchs70}). Furthermore, since $C$ is pure in $\mc{X}$ and is finitely generated (since it is contained in the finitely generated group $\mc{X}$), it detaches as a direct summand:
 $$\mc{X} = C \oplus W$$ for some $W$. We can choose generators $\bar{g} \bar{h}$ of $\mc{X}$ so that $C = \langle \bar{g} \rangle$ and $W = \langle \bar{h} \rangle$. Moreover, we may choose these generators to be linearly independent (since $C$ and $W$ decompose as direct sums of copies of $\mathbb{Z}$). From now on, we will assume that every generating set that we consider is such a linearly independent generating set.

Replace $\bar{c}, a, \bar{y}$ in $\psi$ by the respective linear combinations of $\bar{g} \bar{h}$ and denote the resulting formula by $\phi(\bar{g},\bar{h})$. It is equal to $\phi_0(\bar{g},\bar{h}) \wedge \phi_1(\bar{g},\bar{h})$, where $\phi_0(\bar{g},\bar{h})$ is a conjunction of $\mathbb{Z}$-linear equations and $\phi_1(\bar{g},\bar{h})$ is a conjunction of $\mathbb{Z}$-linear inequations. Write $\bar{h} = (h_1, \ldots, h_k)$ and choose any non-zero $u\in C$. We claim that there is a natural number $m$ such that replacing $h_k$ by $h_k' = mu$ preserves the validity of $\phi$ in $\mc{X}$ (and thus in $\mc{M}$ since $\phi$ is quantifier-free).
 
To see why this is the case, note that $\bar{g},\bar{h}$ is an independent set, and therefore $\phi_0$ must necessarily be a conjunction of trivial equations. If we replace $h_k$ by any other value, $\phi_0$ will still trivially be true. Now we turn to $\phi_1$. Replace $h_k$  in $\phi_1$ by a new indeterminate $z$.  There are only finitely many linear equations that could potentially witness the failure of the formula $\phi_1$ (viewed as a formula of $z$). Since the group is torsion-free, each of these equations has at most one solution (in $z$), since if a solution exists then it can be uniquely expressed as a linear combination of $\bar{g}$ and the rest of the $h_i$ (possibly with rational coefficients). Since $\mc{M}$ and hence $\mc{X}$ is torsion-free,  $\langle u \rangle \cong \mathbb{Z}$ is infinite, and so there is some $m$ such that $mu$ is not a solution to any of these equations. We conclude that replacing $h_k$ by $h_k = mu$  preserves the validity of the formula $\phi$.
We can repeat the process described above for $h_{k-1},h_{k-2},\ldots$, each time replacing them by $h_i' \in C$. So we get a tuple $\bar{h}' \in C$ such that $\phi(\bar{g},\bar{h}')$ holds in $\mc{X}$. 
 
We also note that the replacement of the parameters $\bar{c},a,\bar{w}$ in $\psi$ by $\bar{g},\bar{h}$ described above uniquely induces a replacement of the original parameters $\bar{c}, a, \bar{w}$ in $\psi$ by new parameters $\bar{c}, b, \bar{w}'$ by writing $b$ and $\bar{w}'$ as a linear combination of $\bar{g},\bar{h}'$ with the same coefficients as when we wrote $a$ and $\bar{w}$ as a linear combination of $\bar{g},\bar{h}$. Note that $\bar{c}$ stays untouched since $\bar{c} \in \langle \bar{g} \rangle$. This replacement preserves the validity of $\psi$, since each elementary equation or inequation still holds after the replacement. Since $\bar{h}' \in C$, $b$ is dependent on $\bar{c}$.
\end{proof}

\begin{claim}\label{cl:tfa2}  $\mc{M}$ satisfies Condition G. \end{claim}
\begin{proof}[Proof of Claim] 

We will begin by proving that  independent tuples are locally indistinguishable. Suppose $\bar{u} = (u_0, \ldots, u_n)$ and $\bar{v} = (v_0, \ldots, v_n)$ are independent over $\bar{c}$.  We may assume that $\bar{c}$ contains at least one non-zero element.

Suppose that $\mc{M} \models \exists \bar{y} \psi(\bar{c}, \bar{u},\bar{y})$ where $\psi$ is quantifier-free. Let $\bar{w}$ be a tuple in $\mc{M}$ witnessing the existential quantifier.
  
Let $\mc{X}$ be the least subgroup of $\mc{M}$ that contains $\bar{c}$, $\bar{u}$ and $\bar{w}$. Let $C$ be the smallest pure subgroup of $\mc{X}$ containing $\bar{c}$ and $U$ the smallest pure subgroup of $\mc{X}$ containing $\bar{u}$. The only elements of $U \cap C$ are those which are linearly dependent over both $\bar{c}$ and $\bar{u}$, and since $\bar{u}$ is independent over $\bar{c}$ and $\mc{X}$ is torsion-free, $U \cap C = 0$. Hence we can decompose $\mc{X}$ as
\[ \mc{X} = C \oplus U \oplus W.\]
By a similar argument to the previous claim, using the fact that $\bar{c}$ contains some non-zero element we can replace $\bar{w}$ with some $\bar{w}' \in C \oplus U$. So we may assume that
\[ \mc{X} = C \oplus U.\]

Let $\bar{g}$ and $\bar{h}$ be independent tuples which generate $C$ and $U$ respectively. Moreover, we may assume that $\bar{h} = (h_1,\ldots,h_n)$ and that there are integers $k_i$ such that $u_i = k_i h_i$. To see this, let $U_i$ be the pure closure of $u_i$ in $U$. Since the $u_i$ are independent, for each $i$ we have
\[ U_i \cap \bigoplus_{j \neq i} U_j = 0. \]
Thus $U = U_1 \oplus \cdots \oplus U_n$. Each $U_i$ is isomorphic to $\mathbb{Z}$, so we can take $h_i$ to be a generator for $U_i$.

Define $\tau: \mc{X} \rightarrow C \oplus \langle v_0, \ldots, v_n \rangle $ to be the unique homomorphism such that $\tau(g_i) = g_i$ and $\tau(h_j) = v_j$ for all $i$ and $j$. Obviously $\tau$ is a homomorphism, but note that it is in fact an isomorphism: it is clearly onto, and it is also one-to-one since it maps a linearly independent generating set of $\mc{X}$ to a linearly independent set in the image. Thus $\tau(\bar{u})$ is  independent over $\cl(\bar{c})$ and satisfies $\exists \bar{y} \psi(\bar{c}, \tau(\bar{u}),\bar{y})$ with $\tau(\bar{w})$ being the witness for the $\exists$-quantifier. Finally, since for each $i$ we know that $u_i$ is a multiple of $h_i$, $\tau(u_i)$ is a multiple of $v_i$. Thus $\tau(u_i)$ and $v_i$ are interdependent.
  
\smallskip

We will now explain how we list the independence diagram of a tuple $\bar{c}$. We begin with some preliminary remarks. Given a formula $\phi(\bar{c},\bar{x}) = \exists \bar{y} \psi(\bar{c}, \bar{x}, \bar{y})$ which holds for some $\bar{a}$ independent over $\bar{c}$, then we can find a finitely generated subgroup $G$ of $\mc{M}$ containing $\bar{a}$ and $\bar{c}$ so that $G \models  \exists \bar{y} \psi(\bar{c}, \bar{a}, \bar{y})$.
  
 Using the same argument as above, we can write
 \[ G = C \oplus W,\]
 where $C$ is the pure subgroup of $G$ generated by $\bar{c}$, $W$ is the pure subgroup of $G$ generated by $\bar{a}$, and the tuple $\bar{w}$ witnessing the existential quantifier belongs to $G$. Choose independent generators $\bar{g}$ and $\bar{h}$ of $C$ and $W$ respectively. Then we can write $\bar{c}$ and $\bar{a}$ as $\mathbb{Z}$-linear combinations of $\bar{g}$ and $\bar{h}$. Let $p = |\bar{g}|$ and $q = |\bar{h}|$. Since $\bar{g}$ and $\bar{h}$ are independent, the elements of $G$ are in one-to-one correspondence with the $\mathbb{Z}$-module $\mathbb{Z}^p \oplus \mathbb{Z}^q$ with $\bar{g},\bar{h}$ mapping to the standard basis elements. We call the image of some $g \in G$ in $\mathbb{Z}^p \oplus \mathbb{Z}^q$ the \textit{formal representation} of $g$. Then the formal representations of $\bar{a}$ are $\mathbb{Z}$-linearly independent over the formal representations of $\bar{c}$, and hence over $\mathbb{Z}^p$.

Observe that there is a finite partial subgroup $H \subseteq G$ containing $\bar{c}$ and $\bar{a}$ such that $H \models  \exists \bar{y} \psi(\bar{c}, \bar{a}, \bar{y})$ and such that $H$ contains only elements of the form $$\sum_{i} m_i g_i + \sum_{j} n_j h_j,$$ 
  with $|m_i|, |n_j| \leq k$ for some $k$. By a partial group, we mean that $H$ is not necessarily closed under the group operations, but that the operations on $H$ agree with those on $G$ where possible. By $H \models  \exists \bar{y} \psi(\bar{c}, \bar{a}, \bar{y})$, we mean that there is a tuple $\bar{w}'$ in $H$ such that $\psi(\bar{c}, \bar{a}, \bar{w})$ holds with all of the intermediate terms occurring in $H$ (e.g., if $\psi$ is $c + a - 2w = 0$, then all of the terms $c + a$, $2w$, and so on appear in $H$). This partial group is isomorphic to a direct sum of partial additive groups $\mathbb{Z}$$\upharpoonright$$k$ upon $\{-k, \ldots, -1, 0, 1, \ldots, k\}$, since the $g_i$ and $h_j$ are independent (we believe that the notion of a direct sum is self-explanatory when applied to partial structures).
  
We claim that the formula $ \exists \bar{y} \psi(\bar{c}, \bar{x}, \bar{y})$ is in the independence diagram of $\bar{c}$ if and only if there exists a  finite partial subgroup $H$ containing $\bar{c}$ and a tuple $\bar{a}$ such that $H \models  \exists \bar{y} \psi(\bar{c}, \bar{a}, \bar{y})$ and $H$ has the form as described above,~i.e.: 
  
  (1) $H$ is direct sum of partial subgroups $C$ and $W$ generated by $\bar{g}$ and $\bar{h}$ as above, 
  
  (2) $\bar{c} \in C$, $\bar{a}, \bar{w} \in H$,
  
  (3) both $C$ and $W$ are direct sums of partial groups of the form $\mathbb{Z}$$\upharpoonright$$k$ for some $k$, and 
  
  (4) the formal representations of $\bar{a}$ are linearly independent over the formal representations of~$\bar{c}$. 
  
We have already checked that such a partial subgroup $H$ exists for any  $\exists \bar{y} \psi(\bar{c}, \bar{x}, \bar{y})$ in the independence diagram of $\bar{c}$. Conversely, suppose such a finite partial $H$ exists. Since the dimension of $\mc{M}$ is infinite, we can find a tuple $\bar{u}$ in $\mc{M}$ independent over $\bar{g}$ (and thus, over $\bar{c}$). The map $h_j \rightarrow u_j$ can be uniquely extended, via linear combinations, to an isomorphic embedding $\tau$ of $H$ into $\mathcal{M}$ that fixes $\bar{g}$ componentwise. Since $\psi$ is quantifier-free, we have $\mc{M} \models \psi(\bar{c}, \tau(\bar{a}), \tau(\bar{w}))$ is preserved under this embedding. Furthermore,  since the formal representation of $\bar{a}$ with respect to $\bar{g},\bar{h}$ is the same as that of $\tau(\bar{a})$ with respect to $\bar{g},\bar{u}$, since $\bar{u}$ is  independent over $\bar{g}$, we conclude that the image of  $\bar{a}$ is independent over $\bar{c}$ as desired.
    \end{proof}
It remains to apply Theorem~\ref{main}. \end{proof}

\subsection{Archimedean ordered abelian groups}\label{sec:oag} Recall that an ordered abelian group $A$ is Archimedean if for every non-zero $a,b \in A$ there exists an $m \in \mathbb{Z}$ such that $m|a| > |b|$ and $m|b|> |a|$, where $|x| = x$ if $x>0$  and $|x| = -x$ otherwise. See Kokorin and Kopytov~\cite{KokorinKopytov74} and Fuchs~\cite{Fuchs63} for more algebraic background on ordered groups.

\medskip

\noindent \textbf{Theorem~\ref{th:aoag}.}  \textit{The class of computable Archimedean ordered abelian groups has the  Mal{\textquotesingle}cev property with respect to $\mathbb{Z}$-independence.}

\medskip

\begin{proof} Suppose $\mc{M}$ is a computable  Archimedean ordered abelian group of infinite dimension. We will use properties of the ordered field  $\mathbb{R}$ throughout by embedding $\mc{M}$ into $\mathbb{R}$.

\begin{claim}\label{zozo1}
$\mc{M}$ satisfies Condition B.
\end{claim}

\begin{proof} Suppose $\mc{M} \models \exists \bar{y} \psi(\bar{c}, a, \bar{y})$, where $a$ is $\mathbb{Z}$-independent over $\bar{c}$. We also assume $\bar{c}$ contains at least two linearly independent elements. Fix any tuple $\bar{w}$ witnessing $\exists \bar{y}$.
 As in the proof for torsion-free abelian groups, consider the free abelian group $\mc{X}$ spanned by $\bar{c}, a, \bar{w}$ and  let $C$ be  the least pure subgroup of $\mc{X}$ that contains $\bar{c}$.
We have  $$\mc{X} = C \oplus W$$ (group-theoretically) and thus we can choose generators $\bar{g} \bar{h}$ of $\mc{X}$ so that $C = \langle \bar{g} \rangle$ and $W = \langle \bar{h} \rangle$. Choose a formula $\phi(\bar{g},\bar{h})$ by replacing $\bar{c},a,\bar{w}$ with $\bar{g},\bar{h}$ as we did for torsion-free abelian groups.

We use the well-known fact that  every Archimedean ordered abelian group can be isomorphically embedded into the ordered group of reals $(\mathbb{R}, +, \leq)$~\cite{Holder96}. We identify $\mc{X}$ with its image under this embedding. Let $Y$ be the subset of $\mathbb{R}^n$ isolated by $\phi(\bar{g}, \bar{x})$.  Since $\bar{g},\bar{h}$ are independent, the formula $\phi$ can contain only trivial linear equations in $\bar{g}\bar{h}$. Thus, each component of $\bar{h}$ is contained in $Y$ together with some interval (similar to Claim~\ref{kuku2}). It is also well-known~\cite[Exercise 21 of Section 4 of][]{AliprantisBurkinshaw98} that if $H \leqq (\mathbb{R}, +)$ and $\dim(H) \geq 2$, then $H$ is dense in $(\mathbb{R}, <)$. By our assumption,  $\dim(C) \geq 2$, so we can choose $\bar{h}' \in C$ which are contained in these intervals and hence satisfy $\phi(\bar{g}, \bar{h}')$. Then, as with unordered groups, we can find $b$ and $\bar{w}'$ which are linear combinations of $\bar{g}$ and $\bar{h}$ which satisfy $\psi(\bar{g}, b,\bar{w}')$.
\end{proof}

\begin{claim}\label{zozo2}
$\mc{M}$ satisfies Condition G.
\end{claim}

\begin{proof}
We first show that independent tuples in $\mc{M}$ are locally indistinguishable.  Suppose $\bar{u} = (u_0, \ldots, u_n)$ and $\bar{v} = (v_0, \ldots, v_n)$ are independent over $\bar{c}$. 
We assume that the dimension of $\bar{c}$ is at least~$1$. Let $\theta$ be an existential formula such that $\mc{M} \models \theta(\bar{c}, \bar{u})$. 
We show that there exists a tuple $\bar{z} = (z_0, \ldots, z_n) \in \cl(\bar{c},\bar{v})$ independent over $\bar{c}$ such that $\mc{M} \models \theta(\bar{c}, \bar{z})$.

As we have already noted above, $\mc{M}$ can be identified with a subgroup of $\mathbb{R}$. 
Since $\bar{u}$ is independent over $\bar{c}$, there exists an open neighborhood $U$ of $\bar{u}$ in the respective power of $\mathbb{R}$ such that every tuple from $U$ satisfies $\theta$. The argument is essentially the same as before. We can as before pass to the smallest pure subgroup of $\mc{M}$ containing $\bar{c}, \bar{u}$ and the witnesses for the existential quantifier. Then we re-write the formula in terms of the generators of this finitely generated free abelian group. We then conclude that all the equations in the formula must become trivial after the re-writing, and thus each generator is contained in a definable open interval such that any choices of elements from those intervals satisfy the formula. Since $\bar{u}$ is a linear combination of the generators, it is also contained in a definable open set $U$ around with the property that every tuple from $U$ satisfies $\theta(\bar{c},x)$.

Since $\bar{v}$ is independent over $\bar{c}$, for every $i$ and any non-zero $c \in \cl(\bar{c})$ the set $\{sc +tv_i: s, t \in \mathbb{Z}\}$ is dense in $(\mathbb{R}, <)$ since it has dimension at least 2. Thus, for an arbitrary non-zero $c \in \cl(\bar{c})$ we can find  $s_i, t_i$ such that the tuple $\bar{z} = (z_0, \ldots, z_n)$ with
 $$z_i = s_i c + t_iv_i, \,\,\,\,\, i  = 0, \ldots, n, $$
belongs to $U$. Then $\bar{z}$ satisfies the desired properties.
  
Now we will describe a method of enumerating the independence diagram of $\bar{c} \in \mc{M}$. Every existential formula $\theta(\bar{c}, \bar{x})$ in the independence diagram of $\bar{c}$ is witnessed by a quantifier-free formula $\psi(\bar{c}, \bar{a}, \bar{w})$ in the open diagram of $\mc{M}$, where $\bar{a}$ is independent over $\bar{c}$ and where $\bar{w}$ are the witnesses to the existential quantifier from $\theta$.
 
 We follow the second half of the proof of Claim~\ref{cl:tfa2} closely. We claim that $\theta(\bar{c}, \bar{a})$ is in the independence diagram of $\bar{c}$ if and only if $\psi(\bar{c},\bar{a},\bar{w})$ is satisfied in a finite partial ordered subgroup $H$ of $\mc{M}$ that:
 
  (1) $H$ is direct sum of partial subgroups $C$ and $W$ generated by $\bar{g}$ and $\bar{h}$ as in Claim~\ref{cl:tfa2}, 
  
  (2) $\bar{c} \in C$, $\bar{a}, \bar{w} \in H$,
  
  (3) both $C$ and $W$ are direct sums of partial groups of the form $\mathbb{Z}$$\upharpoonright$$k$ for some $k$, and 
  
  (4) the formal representations of $\bar{a}$ are linearly independent over the formal representations of~$\bar{c}$.

\noindent If the formula $\theta(\bar{c},\bar{x})$ is in the independence diagram of $\bar{c}$, then such an $H$ exists (the proof is the same as for Claim~\ref{cl:tfa2}).

Now suppose such an $H$ exists. Rewrite the formula $\psi(\bar{c},\bar{a},\bar{w})$ in terms of generators of $C$ and $H$. Then all of the linear equations become trivial (since otherwise we would not have $H = \bigoplus_i \mathbb{Z}$$\upharpoonright$$k$). Therefore, 
the new formula isolates a non-empty open set in the respective power of $\mathbb{R}$. Since  $\dim(\mc{M})$ is infinite, using elements of the form $mx+ny$, where $x,y $ are independent over $\bar{c}$,  we can find 
a tuple $\bar{\xi}$ independent over $\bar{c}$ (equivalently, over $C$) that satisfies the formula  (as these elements for a group of dimension at least two). 
We then re-write the formula ``back'', as in the proof of Claim~\ref{cl:tfa2}. In other words, if $C = \langle \bar{g} \rangle$ and $W = \langle \bar{h} \rangle$, then $x_j = \sum_i m_{j,i} g_i + \sum_k n_{j,k} h_k$
is replaced by $z_i = \sum_i m_{j,i} g_i + \sum_k n_{j,k} \xi_k$. 
 The same argument as in Claim~\ref{cl:tfa2} that involves dimensions of formal representations shows that the new solution $\bar{z}$ is $\cl(\bar{c})$-independent.
\end{proof}\renewcommand{\qedsymbol}{}
\end{proof}

\section{Conclusion}
We suspect that our metatheorem holds for many other algebraic classes and classes of theories. For instance,  the theory $T_{\exp}$ of $\mathbb{R}$ as a field with the exponential function would be an appropriate candidate.
 Wilkie and Macintyre \cite{MacintyreWilkie96} showed that, assuming that Schanuel's conjecture\footnote{Schanuel's conjecture is as follows: suppose that $n \geq 1$ and $c_1,\ldots,c_n \in \mathbb{R}$ are linearly independent over $\mathbb{Q}$; then $c_1,\ldots,c_n,e^{c_1},\ldots,e^{c_n}$ have transcendence degree at least $n$ over $\mathbb{Q}$.} is true, $T_{\exp}$ is a decidable theory. This is still unresolved.
Recently, Jones and Servi \cite{JonesServi11} and Miller \cite{Miller} gave examples of decidable theories expanding the theory of real closed fields.  We suspect that the corresponding algebraic classes might have the  Mal{\textquotesingle}cev property.
We also conjecture that some other field-like structures perhaps including algebraically closed valued fields and the like~\cite{HaskellHrushovskiMacpherson08} have the  Mal{\textquotesingle}cev property. 

Although we did not include the case of an ordered abelian group with finitely many Archimedean classes, we conjecture that our methods can be applied to simplify the proof in~\cite{GoncharovLemppSolomon03}. Our ideas may be useful in covering the case of infinitely many Archimedean classes (this is an open problem~\cite{GoncharovLemppSolomon03}), but perhaps some adjustments and new ideas will be necessary. 
 
 We also conjecture that our metatheorem can be applied to arbitrary computable abelian groups with respect to $\mathbb{Z}$-independence (note that torsion elements are ``dependent on themselves''), see \cite{Goncharov80, Khisamiev98}. 

We suspect that our metatheorem has relativized and generalized versions that could be applied to, say, completely decomposable groups~\cite{DowneyMelnikov14,Khisamiev02} and other structures where the notions of independence are not r.i.c.e. but are relatively intrinsically $\Sigma^0_n$ for some $n>1$.

We also note that the spiritually related \emph{$p$-basic tree problem}~\cite{AshKnight00,Melnikov} seems to require new ideas since the corresponding notion of independence is \emph{not} a pregeometry.
Finally, we would like to find some non-trivial applications to \emph{non-commutative} structures.

\bibliography{References}
\bibliographystyle{alpha}

\end{document}